\documentclass[11pt]{article}
\usepackage{amsmath,amsthm,verbatim,amssymb,amsfonts,amscd, graphicx, fancyhdr, enumerate}
\usepackage{parskip}

\usepackage[utf8]{inputenc} 
\usepackage[T1]{fontenc} 

\makeatletter
\def\thm@space@setup{%
	\thm@preskip=\parskip \thm@postskip=0pt
}
\makeatother 

\theoremstyle{plain} \numberwithin{equation}{section}
\newtheorem{theorem}{Theorem}[section]
\newtheorem{corollary}[theorem]{Corollary}

\newtheorem*{conjecture*}{Conjecture}
\newtheorem{lemma}[theorem]{Lemma}

\newtheorem*{question*}{Question} \topmargin-2cm
\theoremstyle{definition}
\newtheorem{definition}[theorem]{Definition}

\newtheorem{remark}[theorem]{Remark}

\usepackage[textheight=23cm, textwidth=16.5cm, footskip=1.0cm]{geometry}

\DeclareMathOperator{\Hessian}{Hess}
\DeclareMathOperator{\distance}{dist}
\DeclareMathOperator{\support}{supp}

\renewcommand{\Im}{\operatorname{Im}}
\renewcommand{\Re}{\operatorname{Re}}
\renewcommand{\i}{\operatorname{\sqrt{-1}}}

\usepackage[
backend=biber,  
natbib=true,
url=false, 
]{biblatex}
\begin{document}
	
	\title{The Diederich--Forn\ae ss index II: for domains of trivial index}

	\author{
		Bingyuan Liu\\ bingyuan@ucr.edu
	}
	
	\date{March 9, 2017}

	\maketitle
	
	\begin{abstract}
We study bounded pseudoconvex domains in complex Euclidean spaces. We find analytical necessary conditions and geometric sufficient conditions for a domain being of trivial Diederich--Forn\ae ss index  (i.e. the index equals to 1). We also connect a differential equation to the index. This reveals how a topological condition affects the solution of the associated differential equation and consequently obstructs the index being trivial. The proofs relies on a new method of study of the complex geometry of the boundary. The method was motivated by geometric analysis of Riemannian manifolds. We also generalize our main theorems under the context of de Rham cohomology.
	\end{abstract}
	
	\section{Introduction}\label{sec0}
	
	Pseudoconvexity is of central importance in modern complex analysis. For example, In 1952, Garabedian--Spencer \citep{GS52} suggested to study the $\bar{\partial}$-Neumann problem on smoothly bounded pseudoconvex domains $\Omega$. After that, the $L^2$ existence and the global regularity for the $\bar{\partial}$-Neumann operator for strongly pseudoconvex domains was solved by Kohn \citep{Ko63}, \citep{Ko64} and H\"{o}rmander \citep{Ho65} in the 1960s. However, the global regularity for (weakly) pseudoconvex domains still lacks a complete answer, even if some sufficient conditions have been found. The Diederich--Forn\ae ss index serves as a refinement of (weakly) pseudoconvex domains and is well-known as a key to the global regularity and many other problems of pseudoconvex domains.
	
	A connection between the Diederich--Forn\ae ss index and global regularity of the $\bar{\partial}$-Neumann operator\slash Bergman projection has been realized more and more precisely since the 1990s. For example, in 1992 Barrett  \citep{Ba92} showed that the Sobolev space $W^k(\overline{\Omega}_\beta)$ is not preserved by the $\bar{\partial}$-Neumann operator\slash Bergman projection when $s\geq\frac{\pi}{\beta-2\pi}$. Here, $\Omega_\beta$ ($\beta>\pi/2$) denotes a $\beta$-worm domain and is known to have non-trivial Diederich--Forn\ae ss indexes (i.e., the index is strictly less than 1). Kohn \citep{Ko99} in 1999 has shown the index is related to orders of the preserved Sobolev space $W^s(\overline{\Omega})$ by $\bar{\partial}$-Neumann operator\slash Bergman projection. He proved the greater the index, the greater order of $W^k(\overline{\Omega})$ can be obtained. Moreover, Berndtsson--Charpentier in \citep{BC00} asserted that $\bar{\partial}$-Neumann operator and Bergman projection preserves $W^k(\overline{\Omega})$ if $s$ is less than a half of the Diederich--Forn\ae ss index. However, an exact relation between the Diederich--Forn\ae ss index and order of $w^k(\overline{\Omega})$ is yet to be found.
	
	In contrast to domains of non-trivial index, domains of trivial index (i.e. the index equals to 1) have many nice properties. These domains have many similarities with the well-understood strongly pseudoconvex domains. Thus, understanding the trivial index is necessary and important. In the 2000s, Forn\ae ss--Herbig (see \citep{FH07} and \citep{FH08}) proved that the domain is of trivial index if it admits a plurisubharmonic defining function on the boundary. Recently, Krantz--Liu--Peloso \citep{KLP16} improved Forn\ae ss--Herbig's condition in $\mathbb{C}^2$ and exhibited more pseudoconvex domains of trivial index (see also \citep{Li16}). 
	
	Just a while ago, an index associated to boundary was found by the author \citep{Li17}. He showed this new index is equivalent to the Diederich--Forn\ae ss index. One feature of this new index is that it is often easier to compute. With this advantage, we discuss a class of domains with the trivial index from the viewpoints of analysis and geometry. 

	Let $\Omega\subset\mathbb{C}^n$ be a bounded pseudoconvex domain with smooth boundary. The \textit{Diederich--Forn\ae ss index} is defined to be the supremum over all exponents $0<\eta<1$ for which there exists a defining function $\rho$ for $\Omega$ such that $-(-\rho)^\eta$ is plurisubharmonic on $\Omega$ (see also \citep{DF77b}). In this paper, we consider the set $\Sigma$ of non-strongly pseudoconvex points and we assume $\Sigma$ is quipped with a complex structure. It is known that $\Sigma$ can be either a complex submanifold or a Levi-foliation. Our first theorems state that $\Omega$ is of trivial index if $\Sigma$ is simply connected in the case of smooth submanifolds or simply connected leafwise in the case of Levi-folations. This condition surprisingly coincides with Boas--Straube's condition which guarantees the global regularity of the $\bar{\partial}$-Neumann operator\slash Bergman projection, even if the motivations and methods between the two results are quite different. The interested reader is referred to Boas--Straube \citep{BS93} and \citep{BS99}. The mentioned theorems are Theorem \ref{2dim}\slash Theorem \ref{3dim} for $\Sigma$ to be complex submanifolds and more generally Theorem \ref{smooth} for $\Sigma$ to be smooth submanifolds. 
	
	To prove the theorems, we first use the equivalent index defined in \citep{Li17} to show the index is 1 if a (real) partial differential equation can be solved in boundary. We then analyze the equation to show that it is a system of partial differential equations that can be decoupled and the compatibility condition always holds locally. Then a global solution exists if $\Omega$ is simply connected. This proves the trivial index. The ideas of proof involve a detailed analysis on local coordinate charts of $\Sigma$ to obtain the equation and a thorough computation to verify the local compatibility condition. 
	
	Conversely, we also study how the trivial index affects solvability of the aforementioned differential equations. Precisely, without assuming simply connectedness of $\Sigma$, we prove that there exists a sequence of functions approaching the solution of the equation in $L^1$ norm if the Diederich--Forn\ae ss index is 1. If the sequence defines a distribution, then the differential equation is solvable in the distributional sense. This theorem also has a version in complex submanifolds and one in Levi-foliations. More conclusions are available in Theorem \ref{preco2} and Corollary \ref{preco3}. The proof is inspired by Caccioppoli's inequality in the field of geometric analysis and contains careful estimates for seeking the solution.
	
	We also modify the classical de Rham cohomology to fit in our context. With this modification, we can reformulate our theorems in the language of cohomology (see Theorem \ref{co2} and Theorem \ref{co3}). Our theorems relate the index with both analytical properties and topological properties of the boundary. The Diederich--Forn\ae ss index can be compared with the Atiyah-Singer index. 
	
	Finally, we prove that if the Levi-flat points form a real curve, then the Diederich--Forn\ae ss index is 1. Indeed, Theorem \ref{newtheorem} extends a theorem in Krantz--Liu--Peloso \citep{KLP16}, where they need the real curve transversal to $T^{(1,0)}\partial\Omega$.
	
	Since the Diederich--Forn\ae ss index has been introduced in 1977, many mathematicians worked on this topic. Here are a few important works we have not mentioned: Adachi--Brinkschulte \citep{AB14}, Demailly \citep{De87}, Fu--Shaw \citep{FS14}, Herbig--McNeal \citep{HM12b}, Harrington \citep{Ha08}, Krantz--Peloso \citep{KP08} and Range \citep{Ra81}.  
	
	We remind the reader that the smoothness in this paper can be extend to $C^3$.
	
	The outline of the paper is as the following: After some preparation in preliminaries, we globally analyze the Diederich--Forn\ae ss index on $\partial\Omega$ in Section \ref{sec3}. Theorems \ref{2dim}, \ref{3dim}, \ref{preco2} and \ref{co3} and their proofs are contained in this section. In Section \ref{sec4}, we relate our result in Section \ref{sec3} to a type of de Rham cohomology. This relates topological properties with analytical properties of $\partial\Omega$. The main theorems in this section are Theorem \ref{co2} and \ref{co3}. In Section \ref{sec5}, we prove two theorems where the set $\Sigma$ of degenerate Levi-forms is assumed to be a real manifold. We prove the Diederich--Forn\ae ss index is 1 under the same assumption of Boas--Straube \citep{BS93}. We also extend a theorem in Krantz--Liu--Peloso \citep{KLP16}. The main theorems in this section are Theorem \ref{smooth} and Theorem \ref{newtheorem}.

	\section{Preliminaries}\label{sec1}
	We first define pseudoconvexity as follows. This is sometimes also called Levi pseudoconvexity.
		\begin{definition}
		Let $\Omega$ be a bounded domain with $C^2$ boundary in $\mathbb{C}^n$. Let \[\delta(z):=\begin{cases}
		-\distance(z, \partial\Omega) & z\in\Omega\\
		\distance(z, \partial\Omega) & \text{otherwise}.
		\end{cases}\]  be the signed distance function of $\Omega$. The domain $\Omega$ is called pseudoconvex if $\Hessian_\delta (L, L)\geq 0$ for all $(1,0)$ tangent vector field $L$ of $\partial\Omega$.
	\end{definition}
	
	Recall that in \citep{Li17}, we refined the notion of Levi pseudoconvexity. Let $\Sigma$ be the set of points with degenerate Levi-forms. Let $N$ be defined as what follows.
	\[N=\frac{1}{\sqrt{\sum_{j=1}^{n}|\frac{\partial \delta}{\partial z_j}|^2}}\sum_{j=1}^{n}\frac{\partial \delta}{\partial \bar{z}_j}\frac{\partial}{\partial z_j}.\] Here, we use the terminology from K\"{a}hler geometry. The $g$ stands for the standard Euclidean metric on $\mathbb{C}^n$ and $\Hessian_\delta(L, N)=g(\nabla_L\nabla \delta, N)$ stands for Hessian with this metric.

	The reader can check that $N$ has the following properties:
	\begin{enumerate}
		\item $\sqrt{2}N$ is unit vector in $\mathbb{C}^n$.
		\item $N+\overline{N}=\nabla\delta$ in $\mathbb{C}^n$.
		\item $N\delta =\frac{1}{2}$.
	\end{enumerate}

	We obtain the following theorem directly from Theorem 2.8 in \citep{Li17}.
	\begin{theorem}[\citep{Li17}]\label{pre}
			Let $\Omega$ be a bounded pseudoconvex domain with smooth boundary in $\mathbb{C}^n$. Let $L$ be an arbitrary smooth $(1,0)$ tangent vector field on $\partial\Omega$. Let \[\Sigma_L:=\lbrace p\in\partial\Omega: \Hessian_\delta(L, L)=0 \enskip\text{at}\enskip p\rbrace.\] The Diederich--Forn\ae ss index for $\Omega$ is 1 if and only if for any $\eta\in(0,1)$, there exists a smooth function $\psi$ which is defined on a neighborhood of $\Sigma$ in $\partial\Omega$, so that on $\Sigma_L$,
		\[
		\left(\frac{1}{1-\eta}-1\right)\left|\frac{1}{2}\overline{L}\psi+\Hessian_\delta(N, L)\right|^2+\frac{1}{2}\left(\frac{1}{2}\Hessian_\psi(L, L)+g(\nabla_L\nabla_{N}\nabla\delta, L)\right)\leq 0,
		\] for all $L$.
	\end{theorem}

 We also need the following lemma which was proved in \citep{Li17}.
\begin{lemma}[\citep{Li17}]\label{basic2}
	Let $\Omega$ be a bounded pseudoconvex domain with smooth boundary in $\mathbb{C}^n$. Suppose $L$ is a $(1,0)$ tangent vector field so that $\Hessian_r(L, L)=0$ at $p\in \partial\Omega$. Assume, $T_j$ for $1\leq j\leq n-2$ are $(1,0)$ tangent vector fields and $L, T_1, T_2, \dots, T_{n-2}$ are orthogonal at $p$. Then $\Hessian_\delta (L, T_j)=0$ for $1\leq j\leq n-2$ at $p$.
\end{lemma}

\section{Global analysis on the Diederich-Forn\ae ss index of $\partial\Omega$}\label{sec3}

We let $\Sigma$ be the set of points with degenerate Levi-form. That is \[\Sigma=\lbrace p\in\partial\Omega: \Hessian_\delta(L, L)=0 \enskip\text{at}\enskip p\rbrace,\] for some non-zero $(1,0)$-tangent vector $L\in T^{(1,0)}\partial\Omega$. In this section, we assume $\Sigma$ admits a geometric structure. In other words, we assume two cases: either $\Sigma$ is a complex manifold with boundary or a Levi-foliation (where each leaf of the foliation is again a complex manifold with boundary). They inherit a K\"{a}hler metric and complex structure from $\mathbb{C}^n$. From now on, we will use $U_\alpha$ denote a coordinate chart of $\Sigma$. 

Let a real function $u$ be in $\mathbb{C}^n$ with coordinates $z_1=x_1+\i y_1, z_2=x_2+\i y_2,\dots, z_n=x+\i y_n$. Given that $u$ solves the equation $\frac{\partial u}{\partial x_j}=f_j$ and $\frac{\partial u}{\partial y_j}=g_j$ for $f_j$ and $g_j$ smooth and for $1\leq j\leq n$, then the following necessary condition holds. It is obtained by a straightforward calculation.

\begin{lemma}\label{basicnoc}
	Let $f_j$ and $g_j$ be smooth real functions in $D\subset\mathbb{C}^n$. Then the equations \[\frac{\partial f_j}{\partial x_i}=\frac{\partial f_i}{\partial x_j},\quad\frac{\partial f_j}{\partial y_i}=\frac{\partial g_i}{\partial x_j}\quad\text{and}\quad\frac{\partial g_j}{\partial y_i}=\frac{\partial g_i}{\partial y_j}\] are equivalent to the following identities:
		\[\frac{\partial h_j}{\partial\bar{z}_i}=\frac{\partial h_i}{\partial \bar{z}_j}\quad\text{and}\quad\frac{\partial h_j}{\partial z_i}=\frac{\partial \bar{h}_i}{\partial \bar{z}_j}\] for all $1\leq i, j\leq n$, where $h_j=f_j+\i g_j$.
\end{lemma}

\begin{lemma}\label{lem1}
Let $\Omega$ be a bounded pseudoconvex domain with smooth boundary in $\mathbb{C}^n$. Let $\Sigma$ be the set of points with degenerate Levi-forms. Assume $\Sigma$ is a complex submanifold with boundary of dimension $m$ in $\partial\Omega$ (with induced complex structure and metric) and the Levi-form $\Hessian_\delta(L, L)\vert_{p\in\Sigma}=0$ if and only if $L_p\in T^{(1,0)}\Sigma$.  Let $\lbrace z_j\rbrace_{j=1}^m$ be local coordinates in a coordinate chart $U_\alpha$. We denote $N_\delta$ with $N$. Then we have the following identities:
\begin{enumerate}
	\item \[\frac{\partial}{\partial z_j}\Hessian_\delta(N, \frac{\partial}{\partial z_i})=\frac{\partial}{\partial \bar{z}_i}\Hessian_\delta(\overline{N}, \frac{\partial}{\partial \bar{z}_j}),\]
	\item \[\frac{\partial}{\partial z_j}\Hessian_\delta(\overline{N}, \frac{\partial}{\partial \bar{z}_i})=\frac{\partial}{\partial z_i}\Hessian_\delta(\overline{N}, \frac{\partial}{\partial \bar{z}_j}).\]
\end{enumerate}
\end{lemma}
\begin{proof}

Let $T, L\in T^{(1,0)}\Sigma$. We first prove $Ng(\nabla_T\nabla\delta, L)=\overline{N}g(\nabla_T\nabla\delta, L)$. For this, we just need to show that \[(N-\overline{N})g(\nabla_T\nabla\delta, L)=0.\] By calculation, we know that
\[\begin{split}
&4g(\nabla_T\nabla\delta, L)\\=&g(\nabla_{T+L}\nabla\delta, T+L)-g(\nabla_{T-L}\nabla\delta, T-L)\\&+\i g(\nabla_{T+\i L}\nabla\delta, T+\i L)-\i g(\nabla_{T-\i L}\nabla\delta, T-\i L).
\end{split}\]
Since $L\pm T$, $L\pm \i T$ are $(1, 0)$-tangent vectors, then we have that \[g(\nabla_{T\pm L}\nabla\delta, T\pm L)=0\] and \[g(\nabla_{T\pm \i L}\nabla\delta, T\pm \i L)=0\] on $\Sigma$. By pseudoconvexity, we have \[(N-\overline{N})g(\nabla_{T\pm L}\nabla\delta, T\pm L)=0\] and \[(N-\overline{N})g(\nabla_{T\pm \i L}\nabla\delta, T\pm \i L)=0.\] The reason is as follows. Consider that $N-\overline{N}$ is tangent to $\partial\Omega$ and $g(\nabla_X\nabla\delta, X)\geq 0$ for $X\in T^{(1,0)}\partial\Omega$. Therefore, $g(\nabla_X\nabla\delta, X)=0$ obtains the minimum. Hence, the derivative should be vanishing. This gives \[(N-\overline{N})g(\nabla_T\nabla\delta, L)=0.\]

With the fact that the curvature tensor vanishes for $\mathbb{C}^n$, we are going to calculate 
\[\begin{split}
&Ng(\nabla_T\nabla\delta, L)\\=&g(\nabla_N\nabla_T\nabla\delta, L)+g(\nabla_T\nabla\delta, \nabla_{\overline{N}}L)\\=&g(\nabla_T\nabla_N\nabla\delta, L)+g(\nabla_{[N, T]}\nabla\delta, L)+g(\nabla_T\nabla\delta, \nabla_{\overline{N}}L).
\end{split}
\]
We are going to compute the last two terms in the equation above. Let $\lbrace\xi_j\rbrace_{j=1}^{n}$ be the orthonormal basis of $T_p^{(1,0)}\mathbb{C}^n$. Here we assume that $\xi_1$ is parallel to $L_p$ and $\xi_n=\sqrt{2}N$.
Then at $p\in\partial\Omega$, \[\begin{split}
	g(\nabla_{[N, T]}\nabla\delta, L)=\sum_{j=1}^{n-1}g([N, T], \xi_j)g(\nabla_{\xi_j}\nabla\delta, L)+2g([N, T], N)g(\nabla_{N}\nabla\delta, L).
\end{split}\]
We observe that $g(\nabla_{\xi_j}\nabla\delta, T)=0$ for all $1\leq j\leq n-1$ by Lemma \ref{basic2}. Moreover \[g([N, T], N)=g([N, T], N+\overline{N})=g([N, T], N)=g([N, T], \nabla\delta)=[N, T]\delta=0.\]

Thus, we have that \[g(\nabla_{[N, T]}\nabla\delta, L)=0.\]

To compute $g(\nabla_T\nabla\delta, \nabla_{\overline{N}}L)$, we have that
\[\begin{split}
&g(\nabla_T\nabla\delta, \nabla_{\overline{N}}L)=\sum_{j}^{n-1}g(\nabla_T\nabla\delta, \xi_j)g(\xi_j, \nabla_{\overline{N}}L )+2g(\nabla_T\nabla\delta, N)g(N, \nabla_{\overline{N}}L)\\=&2g(\nabla_T\nabla\delta, N)(\nabla_N\overline{L})\delta=-2\Hessian_\delta(T, N)\Hessian_\delta(N, L).
\end{split}\]

We obtain that \[Ng(\nabla_T\nabla\delta, L)=g(\nabla_T\nabla_N\nabla\delta, L)-2\Hessian_\delta(N, L)\Hessian_\delta(T, N).\]

 Similarly, we have that
\[\begin{split}
\overline{N}g(\nabla_T\nabla\delta, L)=\overline{N}g(\nabla_{\overline{L}}\nabla\delta, \overline{T})=&g(\nabla_{\overline{N}}\nabla_{\overline{L}}\nabla\delta, \overline{T})+g(\nabla_{\overline{L}}\nabla\delta, \nabla_N\overline{T})\\=&g(\nabla_{\overline{L}}\nabla_{\overline{N}}\nabla\delta, \overline{T})+g(\nabla_{[\overline{N},\overline{T}]}\nabla\delta, \overline{T})+g(\nabla_{\overline{L}}\nabla\delta, \nabla_N\overline{T})\\=&g(\nabla_{\overline{L}}\nabla_{\overline{N}}\nabla\delta, \overline{T})-2\Hessian_\delta(N, L)\Hessian_\delta(T, N).
\end{split}\]

Hence, $Ng(\nabla_T\nabla\delta, L)=\overline{N}g(\nabla_T\nabla\delta, L)$
 implies \[g(\nabla_T\nabla_N\nabla\delta, L)=g(\nabla_{\overline{L}}\nabla_{\overline{N}}\nabla\delta, \overline{T}).\]

Let $T=\frac{\partial}{\partial z_i}$ and $L=\frac{\partial}{\partial z_j}$. Using $\nabla_{\frac{\partial}{\partial \bar{z}_i}}\frac{\partial}{\partial z_j}=0$, we obtain that \[\frac{\partial}{\partial z_i}\Hessian_\delta(N, \frac{\partial}{\partial z_j})=\frac{\partial}{\partial \bar{z}_j}\Hessian_\delta(\overline{N}, \frac{\partial}{\partial \bar{z}_i}).\]

We are going to show the second identity. For this, we compute
\[\begin{split}
Tg(\nabla_L\nabla\delta, N)=&g(\nabla_T\nabla_L\nabla\delta, N)+g(\nabla_L\nabla\delta, \nabla_{\overline{T}}N)\\
=&g(\nabla_L\nabla_T\nabla\delta, N)+g(\nabla_{[T, L]}\nabla\delta, N)+g(\nabla_L\nabla\delta, \nabla_{\overline{T}}N)\\
=&Lg(\nabla_T\nabla\delta, N)-g(\nabla_T\nabla\delta, \nabla_{\overline{L}}N)+g(\nabla_{[T, L]}\nabla\delta, N)+g(\nabla_L\nabla\delta, \nabla_{\overline{T}}N).
\end{split}\]

We also compute \[\begin{split}
&g(\nabla_L\nabla\delta, \nabla_{\overline{T}}N)=g(\nabla_L\nabla\delta, g(\nabla_{\overline{T}}, \sqrt{2}N)\sqrt{2}N)\\=&2g(\nabla_L\nabla\delta, N)g(\nabla_T\overline{N}, \nabla\delta)=-2\Hessian_\delta(L, N)\Hessian_\delta(T, N)\\=&g(\nabla_T\nabla\delta, \nabla_{\overline{L}}N).
\end{split}\]

Thus, we have \[Tg(\nabla_L\nabla\delta, N)=Lg(\nabla_T\nabla\delta, N)+g(\nabla_{[T, L]}\nabla\delta, N).\] Let $T=\frac{\partial}{\partial z_i}$ and $L=\frac{\partial}{\partial z_j}$. We find that $[T, L]=0$ and \[\frac{\partial}{\partial z_i}g(\nabla_{\overline{N}}\nabla\delta, \frac{\partial}{\partial \bar{z}_j})=\frac{\partial}{\partial z_j}g(\nabla_{\overline{N}}\nabla\delta, \frac{\partial}{\partial \bar{z}_i})\] which completes the proof.
\end{proof}

\begin{lemma}\label{lem2}
	Let $\Omega$ be a bounded pseudoconvex domain with smooth boundary in $\mathbb{C}^n$. Let $\Sigma$ be the set of points with degenerate Levi-forms. Assume $\Sigma$ is a complex submanifold with boundary of dimension $m$ in $\partial\Omega$ (with induced complex structure and metric) and the Levi-form $\Hessian_\delta(L, L)\vert_{p\in\Sigma}=0$ if and only if $L_p\in T^{(1,0)}\Sigma$. We consider a coordinate chart $U_\alpha$ of $\Sigma$. Let the coordinates be $\lbrace z_j\rbrace_{j=1}^m$. Then for arbitrary $1\leq i, j\leq m$, we have $\Hessian_\delta(N, \frac{\partial}{\partial z_i})=\frac{\partial}{\partial \bar{z}_i}\phi$ locally in a coordinate chart $U_\alpha$, for a locally defined real function $\phi$. Particularly, we have $\Hessian_\delta(N, L)=\overline{L}\phi$ for arbitrary $L\in T^{(1,0)}U_\alpha$. Moreover, in $U_\alpha$, we have that $g(\nabla_L\nabla_N\nabla\delta, L)=\Hessian_\phi(L, L)$. 
\end{lemma}

\begin{proof}
	Since $U_\alpha$ is a coordinate chart, $U_\alpha$ is simply connected. Let $\Hessian_\delta(N, \frac{\partial}{\partial z_i})=f_i+\i g_i$. To check $\Hessian_\delta(N, \frac{\partial}{\partial z_i})=\frac{\partial}{\partial \bar{z}_i}\phi$, we just need to check the condition \[\frac{\partial f_j}{\partial x_i}=\frac{\partial f_i}{\partial x_j},\quad\frac{\partial f_j}{\partial y_i}=\frac{\partial g_i}{\partial x_j}\quad\text{and}\quad\frac{\partial g_j}{\partial y_i}=\frac{\partial g_i}{\partial y_j}.\] By Lemma \ref{basicnoc}, it is equivalent to check the identities \[\frac{\partial}{\partial z_j}\Hessian_\delta(N, \frac{\partial}{\partial z_i})=\frac{\partial}{\partial \bar{z}_i}\Hessian_\delta(\overline{N}, \frac{\partial}{\partial \bar{z}_j})\] and \[\frac{\partial}{\partial\bar{z}_j}\Hessian_\delta(N, \frac{\partial}{\partial z_i})=\frac{\partial}{\partial\bar{z}_i}\Hessian_\delta(N, \frac{\partial}{\partial z_j}).\] Hence, the existence of solution of $\Hessian_\delta(N, \frac{\partial}{\partial z_i})=\frac{\partial}{\partial \bar{z}_i}\phi$ is proved by the preceding lemma.
	
	For general $L$, we assume $L=\sum_{j=1}^{m}\kappa_j\frac{\partial}{\partial z_j}$, where $\lbrace\kappa_j\rbrace_{j=1}^m$ are complex-valued functions. We obtain that \[g(\nabla_N\nabla\delta, L)=\sum_{j=1}^{m}\bar{\kappa}_jg(\nabla_N\nabla\delta, \frac{\partial}{\partial z_j})=\bar{\kappa}_j\frac{\partial}{\partial\bar{z}_j}\phi=\overline{L}\phi.\]
	
	We now prove $g(\nabla_L\nabla_N\nabla\delta, L)=\Hessian_\phi(L, L)$. For this, we just need to prove \[g(\nabla_{\frac{\partial}{\partial z_j}}\nabla_N\nabla\delta, \frac{\partial}{\partial z_j})=\Hessian_\phi(\frac{\partial}{\partial z_j}, \frac{\partial}{\partial z_j}).\] Here, we again assume that $\lbrace z_j\rbrace_{j=1}^m$ are local coordinates. Then using $\nabla_{\frac{\partial}{\partial z_j}}\frac{\partial}{\partial \bar{z}_j}=0$ we have that
	\[\begin{split}
	&\Hessian_\phi(\nabla_{\frac{\partial}{\partial z_j}}\nabla_N\nabla\delta, \frac{\partial}{\partial z_j})\\=&\frac{\partial}{\partial z_j}(\Hessian_\phi(N, \frac{\partial}{\partial z_j}))\\=&\frac{\partial}{\partial z_j}\frac{\partial}{\partial\bar{z}_j}\phi\\=&\Hessian_\phi(\frac{\partial}{\partial z_j},\frac{\partial}{\partial z_j}).
	\end{split}\]
\end{proof}

We will obtain our first theorem which asserts that if $\Sigma$ is simply connected, then the domain $\Omega$ is of trivial index.

\begin{theorem}\label{2dim}
	Let $\Omega$ be a bounded pseudoconvex domain with smooth boundary in $\mathbb{C}^n$. Let $\Sigma$ be the set of points with degenerate Levi-forms. Assume $\Sigma$ is a simply connected complex submanifold with boundary of dimension $m$ in $\partial\Omega$ (with induced complex structure and metric) and the Levi-form $\Hessian_\delta(L, L)\vert_{p\in\Sigma}=0$ if and only if $L_p\in T^{(1,0)}\Sigma$. Then the Diederich-Forn\ae ss index is $1$.
\end{theorem}

\begin{proof}
		Suppose $\lbrace z_j=x_j+\i y_j\rbrace_{j=1}^m$ are local coordinate of $U_\alpha$. Let $g(\nabla_N\nabla\delta, \frac{\partial}{\partial z_j})=f_j+\i g_j$. Since $\Sigma$ is simply connected, by Lemma \ref{lem1} and \ref{lem2}, \[\phi(z):=\sum_{j=1}^{m}\int_{C} f_jdx_j+g_jdy_j\] is well-defined and smooth in each coordinate $U_\alpha$ of $\Sigma$. Here $C$ is an arbitrary curve starting from a fixed point $a\in\Sigma$ to $z$. One observes that this definition is independent of the choice of coordinate chart and hence is well-defined on $\Sigma$. Therefore, for all smooth $(1,0)$-tangent vector fields $L$ ,we have that $\overline{L}\phi=g(\nabla_N\nabla\delta, L)$ in $\Sigma$. This also gives that $g(\nabla_L\nabla_N\nabla\delta, L)=\Hessian_\delta(L, L)$ in $\Sigma$. Letting $\psi=-2\phi$, we obtain that \[\begin{split}
		&\left(\frac{1}{1-\eta}-1\right)\left|\frac{1}{2}\overline{L}\psi+\Hessian_\delta(N, L)\right|^2+\frac{1}{2}\left(\frac{1}{2}\Hessian_\psi(L, L)+g(\nabla_L\nabla_N\nabla \delta, L)\right)\\=&\left(\frac{1}{1-\eta}-1\right)\left|-\overline{L}\phi+\overline{L}\phi\right|^2+\frac{1}{2}\left(\Hessian_\phi(L, L)+\Hessian_\delta(L, L)\right)=0
		\end{split}\] for all $\eta\in(0,1)$. Thus by Theorem \ref{pre}, if we can extend $\phi$ to a neighborhood of $\Sigma$, the Diederich--Forn\ae ss index of $\Omega$ is 1.

We want to define $\tilde{\phi}$ in the neighborhood of $\Sigma$ in $\partial\Omega$ and require that $\tilde{\phi}=\phi$ on $\Sigma$. For this, we check $\phi$ satisfies the condition of Whitney's extension theorem locally around each point in $\Sigma$. Then we glue the extensions with a partition of unity. Let $\xi_0\in\Sigma$, we consider the coordinate chart $U_{\xi_0}$ of $\xi$ in $\partial\Omega$. Let $\lbrace\xi_1,\xi_2,\dots,\xi_m, \zeta_{m+1}, \dots, \zeta_{n-1}, t\rbrace$ be coordinates for $\partial\Omega$, where $t\in\mathbb{R}$. Here we assume $\lbrace \xi_j\rbrace$ are the coordinates in $\Sigma$. It is enough to extend $\phi$ to the manifold foliated by $\lbrace\frac{\partial}{\partial \xi_j}\rbrace_{j=1}^{m}$. Let $h_j(\xi)=\Hessian_\delta(N, \frac{\partial}{\partial \xi_j})$ for $1\leq j\leq m$. Note that $h_j$ is smooth for all $1\leq j\leq m$. In $\Sigma\cap U_{\xi_0}$, the Taylor expansion at $\xi_1\in\Sigma\cap U_{\xi_0}$ gives,
\[\begin{split}
\phi(\xi_2)=&\phi(\xi_1)+\sum_{|\alpha|+|\beta|\leq s}\frac{(\cfrac{\partial}{\partial\bar{\xi}})^\alpha(\cfrac{\partial}{\partial \xi})^\beta \phi(\xi_1)}{\alpha!\beta!}(\xi-\xi_1)^\beta(\bar{\xi}-\bar{\xi}_1)^\alpha+R_s(\xi,\xi_1)\\=&\phi(\xi_1)+\sum_{|\alpha|+|\beta|\leq s}\frac{(\cfrac{\partial}{\partial\bar{\xi}})^\alpha(\cfrac{\partial}{\partial \xi})^{\beta-(1,0,\dots,0)} h_1(\xi_1)}{\alpha!\beta!}(\xi-\xi_1)^\beta(\bar{\xi}-\bar{\xi}_1)^\alpha\\&+\sum_{|\alpha|+|\beta|\leq s}\frac{(\cfrac{\partial}{\partial\bar{\xi}})^{\alpha-(1,0,\dots,0)}(\cfrac{\partial}{\partial \xi})^{\beta} \bar{h}_1(\xi_1)}{\alpha!\beta!}(\xi-\xi_1)^\beta(\bar{\xi}-\bar{\xi}_1)^\alpha\\&+\dots+\sum_{|\alpha|+|\beta|\leq s}\frac{(\cfrac{\partial}{\partial\bar{\xi}})^{\alpha-(0,0,\dots,1)}(\cfrac{\partial}{\partial \xi})^{\beta} \bar{h}_m(\xi_1)}{\alpha!\beta!}(\xi-\xi_1)^\beta(\bar{\xi}-\bar{\xi}_1)^\alpha+R_s(\xi_1,\xi_2).
\end{split}\]

Since $\lbrace h_j\rbrace_{j=1}^m$ are smooth functions, \[R_s(\xi_1,\xi_2)=o(|\xi_2-\xi_1|^{m-1}),\] when $\xi_1,\xi_2\in\Sigma\cap U_{\xi_0}$ are both close to $\zeta_0$ uniformly. Hence $\phi$ is of class $C^s$ in $\Sigma$ in terms of the $\phi_0$. Similarly, one can see that $\phi$ is of class $C^s$ in $\Sigma$ in terms of the $\phi_\alpha$ for $|\alpha|\leq s$. Hence, by a partition of unity there exists a function $\tilde{\phi}$ in a neighborhood of $\Sigma$ in $\partial\Omega$ so that $\tilde{\phi}=\phi$ on $\Sigma$.
\end{proof}

		In $\mathbb{C}^2$, the theorem can be simplified as follows.
\begin{corollary}
	Let $\Omega$ be a bounded pseudoconvex domain with smooth boundary in $\mathbb{C}^2$ and $\Sigma$ be the Levi-flat sets of $\partial\Omega$. Suppose $\Sigma$ is an open simply connected Riemann surface with boundary. Then the Diederich-Forn\ae ss index is $1$.
\end{corollary}

A Levi-flat set automatically forms a closed set. When it forms a foliation, it is called Levi-foliation. For this topic we refer the reader to \citep{BDD07} and \citep{St10}. Since leaves of Levi-foliation usually have boundary, the standard definition of trivial foliation does not work here. Therefore, we redefine the notion of \textit{trivial Levi-foliation with boundary} in our context.

\begin{definition}
	Let $\Omega$ be a bounded pseudoconvex domain with smooth boundary in $\mathbb{C}^n$. Let $\Sigma$ be a Levi-flat set of $\partial\Omega$. The set $\Sigma$ is said to be a \textit{trivial Levi-foliation with boundary} if there exists a smooth map $\mathcal{F}$ defined on a neighborhood $U$ of $\Sigma$ and $\mathcal{F}$ preserves the leafwise complex structure so that $\mathcal{F}(\Sigma)\subset \mathbb{B}^m\times[0,1]=\mathcal{F}(U)$ and each leaf of $\Sigma$ is mapped into $\mathbb{B}^m_t=\mathbb{B}^m\times\lbrace t\rbrace$ for some $t$, where $\mathbb{B}^m$ is the $m$-(complex) dimensional unit ball. We call $\Sigma$ a \textit{trivial Levi-foliation with boundary of simply connected $m$-dimensional leaves} if $\Sigma$ is a trivial Levi-foliation with boundary and each leaf of $\Sigma$ is simply connected.
\end{definition}

\begin{theorem}\label{3dim}
		Let $\Omega$ be a bounded pseudoconvex domain with smooth boundary in $\mathbb{C}^n$. Let $\Sigma$ be the set of points with degenerate Levi-forms. Suppose $\Sigma$ is a trivial Levi-foliation with boundary of simply connected $m$-dimensional leaves (with induced complex structure and metric) and the Levi-form $\Hessian_\delta(L, L)\vert_{p\in\Sigma}=0$ if and only if $L_p\in T^{(1,0)}\Sigma_p$, where $\Sigma_p$ is the leaf of $p$. Then the Diederich-Forn\ae ss index is $1$.
\end{theorem}
\begin{proof}
	The proof is very similar to the proof of Theorem \ref{2dim}. Without loss of generality, we work on $\mathcal{F}(\Sigma)\subset\mathcal{F}(U)$. We let the coordinates be $\lbrace z_1, z_2,...,z_m, t\rbrace$ where $z_j=x_j+\i y_j$ for $1\leq j\leq m$. We know that each leaf $\Sigma_t$ of $\mathcal{F}(\Sigma)$ is included in $\mathbb{B}^m_t$. Let $g(\nabla_N\nabla\delta, \frac{\partial}{\partial z_j})=f_j+\i g_j$. We define \[\phi(z):=\sum_{j=1}^{m}\int_{C} f_jdx_j+g_jdy_j\] in $\mathbb{B}^m\times[0,1]$. Here $C$ is the line segment in $\mathbb{B}^m_t$ connecting $(0,0,...,0,t)$ and $(z_1,z_2,...,z_m, t)$. We can see this is globally defined in $\mathcal{F}(U)$. It is not hard to see that on $\mathcal{F}(\Sigma)$, \[\frac{\partial}{\partial\bar{z}_j}\phi=g(\nabla_N\nabla\delta, \frac{\partial}{\partial z_j}).\] 
	
	By the similar argument as in Theorem \ref{2dim}, we complete the proof.
\end{proof}

For the following paragraphs, we will consider the relation of the equation $\overline{L}\phi=\Hessian_\delta(N, L)$ and the Diederich--Forn\ae ss index. But before that, we need a definition to explain the idea easily.

\begin{definition}
	Let $\lbrace\psi_n\rbrace_{n=1}^\infty$ be a sequence of smooth functions defined on a neighborhood of $\Sigma$. We say $\lbrace\psi_n\rbrace_{n=1}^\infty$ \textit{defines a distribution on the interior $\mathring{\Sigma}$} if $\psi_n$ converges to a distribution $T$ in the weak-* topology on $\mathcal{D}'(\mathring{\Sigma})$. In other words, for any test function $\omega\in C_c^\infty(\mathring{\Sigma})$, we have $<\psi_n, \omega>\rightarrow<T, \omega>$ as $n\rightarrow\infty$.
\end{definition}

\begin{remark}
	It is not hard to see that \[<\partial_\alpha\psi_n, \omega>=-<\psi_n, \partial_\alpha\omega>\rightarrow-<T,\partial_\alpha\omega>=<\partial_\alpha T, \omega>,\] as $n\rightarrow\infty$. This means $\lbrace\partial_\alpha\psi_n\rbrace_{n=1}^\infty$ defines $\partial_\alpha T$.
\end{remark}

For the following, we use the notation: $dV=dx_1\wedge dy_1\wedge\cdots\wedge dx_m\wedge dy_m$.

\begin{lemma}\label{prep}
	Let $\Sigma$ be a complex submanifold with boundary of dimension $m$ in $\partial\Omega$ (with induced complex structure and metric). Let $p\in\Sigma$ and $U_p\subset\Sigma$ is a coordinate chart of $p\in U_p$ with a local coordinate $\lbrace\frac{\partial}{\partial z_j}\rbrace_{j=1}^m$. Assume there is a real function $f$ defined on a neighborhood of $\Sigma$ in $\partial\Omega$ satisfying $n|\frac{\partial}{\partial \bar{z}_j}f|^2+\Hessian_f(\frac{\partial}{\partial z_j}, \frac{\partial}{\partial z_j})\leq 0$ on $U_p$ for some $1\leq j\leq m$ and some $n\in\mathbb{Z}^+$. Then there exists an open set $W_p$ containing $p\in S$ so that  $W_p\subset\overline{W_p}\subset U_p$, and \[\int_{W_p} |\frac{\partial}{\partial \bar{z}_j} f|^2\, dV<\frac{C(U_p)}{n^2},\] where $C(U_p)$ is a constant only depending on $U_p$.
\end{lemma}
\begin{proof}
	Let $u_1=e^{-nf}$ and $u_2=e^{nf}$. We compute 
	\[\frac{\partial}{\partial z_j} u_1\cdot \frac{\partial}{\partial\bar{z}_j} u_2=-n^2e^{-nf}e^{nf}|\frac{\partial}{\partial\bar{z}_j} f|^2=-n^2|\frac{\partial}{\partial\bar{z}_j} f|^2,\]
	and 
	\[u_1\frac{\partial}{\partial\bar{z}_j} u_2=ne^{-nf}e^{nf}\frac{\partial}{\partial\bar{z}_j} f=n\frac{\partial}{\partial\bar{z}_j} f.\]
	
	Observe that \[u_1\cdot\Hessian_{u_2}(\frac{\partial}{\partial z_j}, \frac{\partial}{\partial z_j})=n\left(n\left|\frac{\partial}{\partial\bar{z}_j} f\right|^2+\Hessian_f(\frac{\partial}{\partial z_j}, \frac{\partial}{\partial z_j})\right)\leq 0\] in $U_p$ which is a neighborhood of $p$. Let $W_p\subset V_p$ be two smaller neighborhood of $p$ so that $W_p\subset\overline{W_p}\subset V_p\subset\overline{V_p}\subset U_p$. Choose a smooth function $\chi(q)$, $q\in U_p$ so that 
	\[\chi(q):=\begin{cases}
	1 & q\in W_p\\
	0 & q\in U_p\backslash V_p.
	\end{cases}\] 
	
	Let $\xi=(\xi_1, \xi_2,\dots,\xi_m)\in U_p$ and $S^\xi_j=\lbrace z=(z_1,z_2,\dots, z_m)\in U_p: z_k=\xi_k \enskip\text{if}\enskip k\neq j\rbrace$ be 2-(real) dimensional cross sections of $U_p$.
	Observe that \[\begin{split}
	0\geq&\int_{U_p\cap S^\xi_j}\chi^2u_1\Hessian_{u_2}(\frac{\partial}{\partial z_j}, \frac{\partial}{\partial z_j})\, dx_j\wedge dy_j\\=&-\int_{U_p\cap S^\xi_j}\frac{\partial}{\partial z_j}(\chi^2u_1)\cdot\frac{\partial}{\partial\bar{z}_j} u_2\, dx_j\wedge dy_j\\=&-\int_{U_p\cap S^\xi_j}2\chi u_1\frac{\partial}{\partial z_j}\chi\cdot\frac{\partial}{\partial\bar{z}_j} u_2\, dx_j\wedge dy_j-\int_{U_p\cap S^\xi_j}\chi^2\frac{\partial}{\partial z_j} u_1 \cdot\frac{\partial}{\partial\bar{z}_j} u_2\, dx_j\wedge dy_j\\\geq&-\left(\int_{U_p\cap S^\xi_j} 4|\frac{\partial}{\partial z_j} \chi|^2\, dx_j\wedge dy_j\right)^{1/2}\left(\int_{U_p\cap S^\xi_j} \chi^2 |u_1\frac{\partial}{\partial\bar{z}_j}  u_2|^2\, dx_j\wedge dy_j\right)^{1/2}\\&-\int_{U_p\cap S^\xi_j}\chi^2\frac{\partial}{\partial z_j} u_1\cdot\frac{\partial}{\partial \bar{z}_j} u_2\, dx_j\wedge dy_j.
	\end{split}\]
	
	Hence, we obtain that \[\left(\int_{U_p\cap S^\xi_j} 4|\frac{\partial}{\partial z_j}\chi|^2\, dV\right)^{1/2}\left(\int_{U_p\cap S^\xi_j} n^2\chi^2 |\frac{\partial}{\partial \bar{z}_j} f|^2\, dV\right)^{1/2}\geq\int_{U_p\cap S^\xi_j} n^2\chi^2|\frac{\partial}{\partial \bar{z}_j} f|^2\, dV.\]
	
	Since $\xi$ is arbitrary, we obtain that \[\int_{U_p} 4|\frac{\partial}{\partial z_j}\chi|^2\, dV\geq\int_{U_p} n^2\chi^2|\frac{\partial}{\partial \bar{z}_j} f|^2\, dV\geq\int_{W_p} n^2|\frac{\partial}{\partial z_j} f|^2\, dV.\] We define $C(U_p)=\int_{U_p} 4|\frac{\partial}{\partial z_j}\chi|^2\, dV$ which only depends on $U_p$. We obtain that 
	\[\int_{W_p} |\frac{\partial}{\partial \bar{z}_j} f|^2\, dV\leq\frac{C(U_p)}{n^2}.\]
	
\end{proof}

\begin{theorem} \label{preco2}
	Let $\Omega$ be a bounded pseudoconvex domain with smooth boundary in $\mathbb{C}^n$. Let $\Sigma$ be the set of points with degenerate Levi-forms. Assume $\Sigma$ is a complex submanifold with boundary of dimension $m$ in $\partial\Omega$ (with induced complex structure and metric) and the Levi-form $\Hessian_\delta(L, L)\vert_{p\in\Sigma}=0$ if and only if $L_p\in T^{(1,0)}\Sigma$. Suppose the Diederich--Forn\ae ss index is 1 and $L$ is a $(1,0)$-tangent vector field of $\Sigma$. Then for any compact $\overline{U}\subset\mathring{\Sigma}$, there exists a family $\lbrace\psi_n\rbrace_{n=1}^\infty$ so that
\[\int_{\overline{U}} \left|\frac{1}{2}\overline{L} \psi_n+\Hessian_\delta(N, L)\right|\, dV\rightarrow 0,\] as $n\rightarrow\infty$. Moreover, if $\lbrace\psi_n\rbrace_{n=1}^\infty$ defines a distribution $\psi$ on $\mathring{\Sigma}$, then $\overline{L} \psi=-2\Hessian_\delta(N, L)$ in the distributional sense. Conversely, if there is a smooth function $\psi$ defined on a neighborhood of $\Sigma$, so that $\overline{L} \psi=-2\Hessian_\delta(N, L)$ in $\Sigma$, then the Diederich--Forn\ae ss index is 1.
\end{theorem}

\begin{proof}
	Suppose the Diederich--Forn\ae ss index for $\Omega$ is 1. Then the Diederich--Forn\ae ss index for $\partial\Omega$ is 1. Fix $p\in\mathring{\Sigma}$. Let $U_p$ be a coordinate chart with local coordinates $\lbrace z_j\rbrace_{j=1}^{m}$. Thus, by Lemma \ref{lem2}, we obtain that \[\Hessian_\delta(N, \frac{\partial}{\partial z_j})=\frac{\partial}{\partial \bar{z}_j}\phi\quad\text{and} \quad g(\nabla_{\frac{\partial}{\partial z_j}}\nabla_N\nabla\delta, \frac{\partial}{\partial z_j})=\Hessian_\phi(\frac{\partial}{\partial z_j}, \frac{\partial}{\partial z_j})\] in $U_p$ for all $1\leq j\leq m$. In $U_p$, by Theorem \ref{pre}, we know that for any $\eta\in(0,1)$, there exists a $\psi$ defined in a neighborhood of $\Sigma$ so that\[
	\left(\frac{1}{1-\eta}-1\right)\left|\frac{\partial}{\partial \bar{z}_j} \left(\frac{\psi}{2}+\phi\right)\right|^2+\frac{1}{2}\Hessian_{\frac{\psi}{2}+\phi}\left(\frac{\partial}{\partial z_j}, \frac{\partial}{\partial z_j}\right)\leq 0,
	\] for all $1\leq j\leq m$.
	
	We let $n=2(\frac{1}{1-\eta}-1)$, and define $f_n$ to be $\frac{\psi}{2}+\phi$. Hence, we obtain \[
	n\left|\frac{\partial}{\partial\bar{z}_j} f_n\right|^2+\Hessian_{f_n}\left(\frac{\partial}{\partial z_j}, \frac{\partial}{\partial z_j}\right)\leq0.\] Remember that $f_n$ is only defined in $U_p$. By the preceding lemma, there exists $W_p$ so that  $W_p\subset\overline{W_p}\subset U_p$, and \[\int_{W_p} |\frac{\partial}{\partial\bar{z}_j} f_n|^2\, dV<\frac{C(U_p)}{n^2},\] where $C(U_p)$ is a constant only depending on $U_p$. Let $\eta\rightarrow 1$, and then $n\rightarrow\infty$. This means \[\int_{W_p} |\frac{\partial}{\partial\bar{z}_j}  f_n|^2\, dV\rightarrow 0,\] which implies \[\int_{W_p} |\frac{\partial}{\partial\bar{z}_j}  f_n|\, dV\rightarrow 0,\] for all $j$ because $\mathring{\Sigma}$ is of finite measure.

	Let $\overline{U}$ be covered by $W_p$ finitely. We have that
	\[\int_{\overline{U}} |\frac{1}{2}\overline{L}\psi_n+\Hessian_\delta(N, L)|\, dV=\sum\int_{W_p} |\frac{1}{2}\overline{L}\psi_n+\overline{L}\phi|\, dV\leq \sum\int_{W_p} |\frac{\partial}{\partial\bar{z}_j}  f_n|\, dV\rightarrow 0.\]
	
	Suppose $\lbrace\psi_n\rbrace_{n=1}^\infty$ defines a distribution $\psi$. Let $\omega\in\mathcal{D}'(\mathring{\Sigma})$ be a test function with $\support(\omega)\subset \cup_pW_p$. Since $\support(\omega)$ is compact, then we can choose finite many $W_p$. On each $W_p$, we have that
	\[\int_{W_p} |\frac{1}{2}\overline{L}\psi_n+\Hessian_\delta(N, L)|\, dV\rightarrow 0.\] By a partition of unity, we have that \[\int_{\mathring{\Sigma}}|\frac{1}{2}\overline{L}\psi_n+\Hessian_\delta(N, L)||\omega|\,dV\rightarrow 0\] which means \[-\int_{\mathring{\Sigma}} \psi_n\cdot L\omega\, dV\rightarrow \int_{\mathring{\Sigma}} -2\Hessian_\delta(N, L)\omega\, dV.\] On the other hand, \[-\int_{\mathring{\Sigma}} \psi_n\cdot L\omega\, dV\rightarrow <\overline{L}\psi, \omega>.\] As a result, \[\overline{L}\psi=-2\Hessian_\delta(N, L)\]  in $\mathring{\Sigma}$ in the distributional sense.
	
	Finally the last statement follows from Lemma \ref{lem2}
\end{proof}

\begin{remark}\label{nolim}
	In general $\psi_n$ might not define a distribution. For example, suppose that $\overline{L}\phi=\Hessian_\delta(N, L)$. We let $\psi_n=-2\phi+n$ and observe that \[n\left|\frac{1}{2}\overline{L}\psi_n+g(\nabla_N\nabla\delta, L)\right|^2+\frac{1}{4}\Hessian_{\psi_n}(L, L)+\frac{1}{2}g(\nabla_L\nabla_N\nabla\delta, L)=0.\] Those $\psi_n$ satisfy the condition of the Diederich--Forn\ae ss index is 1 but they diverge to $\infty$ everywhere. Thus, we need to assume that $\lbrace\psi_n\rbrace_{n=1}^\infty$ defines a distribution in the preceding theorem.
\end{remark}

\begin{corollary}\label{preco3}	
	Let $\Omega$ be a bounded pseudoconvex domain with smooth boundary in $\mathbb{C}^n$. Let $\Sigma$ be the set of points with degenerate Levi-forms. Assume $\Sigma$ is a Levi-foliation, i.e., a  $2m+1$-dimensional (real) submanifold with boundary in $\partial\Omega$ foliated by complex submanifolds with boundary of complex dimension $m$ in $\partial\Omega$ (with induced complex structure and metric). Assume the Levi-form $\Hessian_\delta(L, L)\vert_{p\in\Sigma}=0$ if and only if $L_p\in T^{(1,0)}\Sigma_p$, where $\Sigma_p$ is the leaf of $p$. Suppose the Diederich--Forn\ae ss index is 1. Let $L$ be a $(1,0)$-tangent complex vector field of $\Sigma$. Then for any compact subset $\overline{U}\subset\mathring{\Sigma}$, there exists a sequence of smooth functions $\psi_n$ such that \[\int_{\overline{U}} \left|\frac{1}{2}\overline{L} \psi_n+\Hessian_\delta(N, L)\right|\, dV\rightarrow 0,\] as $n\rightarrow\infty$. Here the $dV$ is the leafwise volume form. Moreover, if $\lbrace\psi_n\rbrace_{n=1}^\infty$ defines a distribution $\psi$ on $\mathring{\Sigma}$, then $\overline{L} \psi=-2\Hessian_\delta(N, L)$ in the distributional sense. Conversely, if there is a smooth function $\psi$ defined on a neighborhood of $\Sigma$, so that $\overline{L} \psi=-2\Hessian_\delta(N, L)$ in $\Sigma$, then the Diederich--Forn\ae ss index is 1.
\end{corollary}

	\section{The cohomology related to the index}\label{sec4}
	
	Suppose $\Sigma$ is a complex manifold with boundary or is a Levi-foliation. We define a 1-from on $\Sigma$. 
	\begin{definition}
		Let $\Sigma$ be a complex manifold with boundary of dimension $m$. Assume the Levi-form $\Hessian_\delta(L, L)\vert_{p\in\Sigma}=0$ if and only if $L_p\in T^{(1,0)}\Sigma_p$, where $\Sigma_p$ is the leaf of $p$. Let $p\in\Sigma$ and $U_p$ be a coordinate chart containing $p$. We define the \textit{1-form $\theta$} as what follows: \[\begin{split}
		\theta:=&\sum_{j=1}^{m}\left(\Re\Hessian_\delta(N, \frac{\partial}{\partial z_j})\right)dx_j+\left(\Im\Hessian_\delta(N, \frac{\partial}{\partial z_j})\right)dy_j\\=&\sum_{j=1}^{m}\Re\left(\Hessian_\delta(N, \frac{\partial}{\partial z_j})d\bar{z}_j\right)
		\end{split}\]
	\end{definition}

	We observe that $\theta$ is independent of coordinates. Moreover, by Lemma \ref{lem1} and \ref{lem2}, we know by calculating the local representation of $\theta$ that $d\theta=0$. In other words, $\theta$ is a closed 1-form on $\Sigma$. Once $\theta$ is exact, that means that there will be a real function (0-form) $\phi$ defined on $\Sigma$ so that $\frac{\partial}{\partial \bar{z}_j}\phi=\Hessian_\delta(N, \frac{\partial}{\partial z_j})$ on $\Sigma$. This motivates us to reformulate or extend Theorem \ref{2dim} and \ref{3dim} with the cohomology language. However, we need to require that $\phi$ is defined in a neighborhood of $\Sigma$ to fit in Theorem \ref{pre}. For this, we modify the de Rham cohomology as what follows. 
	
	Let $C^\infty(\Sigma, \Lambda^k)$ denote the $k$-forms which are smooth up to the boundary of $\Sigma$. We define,
	\[\mathfrak{C}^\infty(\Sigma, \Lambda^k):=\lbrace\theta\in C^\infty(\Sigma, \Lambda^k): \theta \text{  can be extended to a neighborhood of $\Sigma$} \rbrace.\] We denote the two forms $\theta_1=\theta_2$ if $\theta_1=\theta_2$ on $\Sigma$, even if they have different extension beyond $\Sigma$. It is clear that $\mathfrak{C}^\infty(\Sigma, \Lambda^k)$ is a subspace of $C^\infty(\Sigma, \Lambda^k)$. We say a $k$-form $\theta$ is closed if $d\theta=0$ on $\Sigma$. We also define our modified de Rham cohomology. \[\mathcal{H}^k(\Sigma):=\lbrace \theta\in\mathfrak{C}^\infty(\Sigma, \Lambda^k): d\theta=0 \rbrace\slash d\mathfrak{C}^\infty(\Sigma, \Lambda^{k-1}).\] By the similar argument of Theorem \ref{preco2}, we obtain the following theorem.
	
		\begin{theorem}\label{co2}
		Let $\Omega$ be a bounded pseudoconvex domain with smooth boundary in $\mathbb{C}^n$. Let $\Sigma$ be the set of points with degenerate Levi-forms. Assume $\Sigma$ is a complex manifold with boundary of dimension $m$ (with induced complex structure and metric) and the Levi-form $\Hessian_\delta(L, L)\vert_{p\in\Sigma}=0$ if and only if $L_p\in T^{(1,0)}\Sigma$. Let $\theta$ be defined above and $[\theta]=[0]$ in $\mathcal{H}^k(\Sigma)$. Then the Diederich--Forn\ae ss index of $\Omega$ is 1.
	\end{theorem}
		In the theory of foliations, the leafwise de Rham cohomology is a well-known tool used to study the topology of leaves. The reader can read \citep{CC00} for the detail. Roughly speaking, in the leafwise deRham cohomology, the exterior derivative in (regular) de Rham cohomology is replaced by leafwise exterior derivative. Moreover, the smooth $k$-forms are replaced by the $k$-forms along leaves $\mathcal{A}^\infty(\Sigma, \Lambda^k)$.
	
	We modify the leaf de Rham cohomology as what follows. Let \[\mathfrak{A}^\infty(\Sigma, \Lambda^k):=\lbrace\theta\in \mathcal{A}^\infty(\Sigma, \Lambda^k): \theta \text{  can be extended to a neighborhood of $\Sigma$} \rbrace.\] We can see that $\mathfrak{A}^\infty(\Sigma, \Lambda^k)$ is a subspace of $\mathcal{A}^\infty(\Sigma, \Lambda^k)$.
	
	We also define the relative de Rham cohomology \[\mathfrak{H}^k(\Sigma):=\lbrace \theta\in\mathfrak{A}^\infty(\Sigma, \Lambda^k): d\theta=0 \rbrace\slash d\mathfrak{A}^\infty(\Sigma, \Lambda^{k-1}).\] Here $d$ should be understood as the exterior derivative along leaves.
	
	By Corollary \ref{preco3}, we obtain the following theorem.
	
	\begin{theorem}\label{co3}
		
			Let $\Omega$ be a bounded pseudoconvex domain with smooth boundary in $\mathbb{C}^n$. Let $\Sigma$ be the set of points with degenerate Levi-forms. Assume $\Sigma$ is a Levi-foliation, i.e., an open ($2m+1$-dimensional) real submanifold (with boundary) of $\partial\Omega$ foliated by complex manifold (with boundary) of complex dimension $m$ (with induced complex structure and metric) and the Levi-form $\Hessian_\delta(L, L)\vert_{p\in\Sigma}=0$ if and only if $L_p\in T^{(1,0)}\Sigma_p$, where $\Sigma_p$ is the leaf of $p$.
			Let $\theta$ be defined above for all leaves and $[\theta]=[0]$ and $[\theta]=[0]$ in $\mathfrak{H}^k(\Sigma)$. Then the Diederich--Forn\ae ss index of $\Omega$ is 1.
	\end{theorem}

\section{Extension to smooth manifolds}\label{sec5}
In this section, we extend the theorems in Section \ref{sec4}. We only assume that $\Sigma$ is a smooth manifold. We denote $N_\delta-\overline{N}_\delta$ by $\nu$. It is not hard to see that $J(\nabla\delta)=\i\nu$,  where $J$ is the complex structure in $\mathbb{C}^n$. Since $\Sigma$ will not be assumed to be a complex manifold with boundary, we in general do not have coordinates $\lbrace z_j\rbrace_{j=1}^m$. For the following, we assume $x_j$ for some $j$ is one of local coordinates in a chart $U_\alpha$. We also define $\frac{\partial}{\partial y_j}$ so that \[J(\frac{\partial}{\partial x_j})=\frac{\partial}{\partial y_j}\qquad\text{and}\qquad J(\frac{\partial}{\partial y_j})=-\frac{\partial}{\partial x_j}.\] Consequently, $\frac{\partial}{\partial z_j}=\frac{1}{2}(\frac{\partial}{\partial x_j}-\i\frac{\partial}{\partial y_j})$ is defined. We remind the reader that $\frac{\partial}{\partial y_j}$ might not tangent to $\Sigma$.

\begin{lemma}
	Let $\Omega$ be a bounded pseudoconvex domain with smooth boundary in $\mathbb{C}^n$. Let $\Sigma$ be the set of points with degenerate Levi-forms. Assume $\Sigma$ is a smooth submanifold with boundary of dimension $m$ in $\partial\Omega$ (with induced metric). Suppose that the Levi-form $\Hessian_\delta(\frac{\partial}{\partial z_j}, \frac{\partial}{\partial z_j})\vert_{p\in\Sigma}=0$, where $x_j$ is one of local coordinates in chart $U_\alpha$ of $\Sigma$.  We denote $N_\delta$ with $N$. Then we have the following identities:
	\[\Re \left(g(\nabla_N\nabla\delta, \frac{\partial}{\partial z_j})\right)=\frac{1}{4}g(\nabla_\nu\nu, \frac{\partial}{\partial x_j})\qquad\text{and}\qquad \Im\left(g(\nabla_N\nabla\delta, \frac{\partial}{\partial y_j})\right)=\frac{1}{4}g(\nabla_\nu\nu, \frac{\partial}{\partial y_j}).\]
	Moreover, the following holds:
	\[\begin{split}
	\frac{\partial}{\partial z_j} g(\nabla_N\nabla\delta, \frac{\partial}{\partial z_j})=&\frac{1}{8}\left(\frac{\partial}{\partial x_j}g(\nabla_\nu\nu, \frac{\partial}{\partial x_j})+\frac{\partial}{\partial y_j}g(\nabla_\nu\nu, \frac{\partial}{\partial y_j})\right)\\=&\frac{1}{8}\left(g(\nabla_{\frac{\partial}{\partial x_j}}\nabla_\nu\nu, \frac{\partial}{\partial x_j})+g(\nabla_{\frac{\partial}{\partial y_j}}\nabla_\nu\nu, \frac{\partial}{\partial y_j})\right).
	\end{split}
	\]
\end{lemma}

\begin{proof}
	We know that \[\begin{split}
	2\Re\left(g(\nabla_N\nabla\delta, \frac{\partial}{\partial z_j})\right)=&g(\nabla_N\nabla\delta, \frac{\partial}{\partial z_j})+g(\nabla_{\overline{N}}\nabla\delta, \frac{\partial}{\partial\bar{z}_j})\\=&g(\nabla_N\nabla\delta, \frac{\partial}{\partial z_j})+g(\nabla_{\frac{\partial}{\partial z_j}}\nabla\delta, N)\\=&g(\nabla_N\nabla\delta, \frac{\partial}{\partial z_j})-g(N, \nabla_{\frac{\partial}{\partial\bar{z}_j}}\nabla\delta).
	\end{split}\]
	The last equality is because that $\frac{\partial}{\partial z_j}(N, N)=0$. Thus, 
	\[\begin{split}
	2\Re\left(g(\nabla_N\nabla\delta, \frac{\partial}{\partial z_j})\right)=&g(\nabla_N\nabla\delta, \frac{\partial}{\partial z_j})-g( \nabla_{N}\nabla\delta, \frac{\partial}{\partial\bar{z}_j})\\=&g(\nabla_N\nabla\delta, -\i\frac{\partial}{\partial y_j})\\=&g(\nabla_NJ(\nabla\delta), -\i J\frac{\partial}{\partial y_j})\\=&g(\nabla_N\nu, \frac{\partial}{\partial x_j}).
	\end{split}\]
	Since $g(\nabla_N\nu, \frac{\partial}{\partial x_j})$ is real, \[g(\nabla_N\nu, \frac{\partial}{\partial x_j})=g(\nabla_{\overline{N}}\overline{\nu}, \frac{\partial}{\partial x_j})=-g(\nabla_{\overline{N}}\nu, \frac{\partial}{\partial x_j}).\] Hence, we obtain that \[g(\nabla_N\nu, \frac{\partial}{\partial x_j})=\frac{1}{2}\left(g(\nabla_N\nu, \frac{\partial}{\partial x_j})-g(\nabla_{\overline{N}}\nu, \frac{\partial}{\partial x_j})\right)=\frac{1}{2}g(\nabla_\nu\nu,\frac{\partial}{\partial x_j}).\] This proves that \[\Re \left(g(\nabla_N\nabla\delta, \frac{\partial}{\partial z_j})\right)=\frac{1}{4}g(\nabla_\nu\nu, \frac{\partial}{\partial x_j}).\] Similarly, one can show that \[\Im\left(g(\nabla_N\nabla\delta, \frac{\partial}{\partial y_j})\right)=\frac{1}{4}g(\nabla_\nu\nu, \frac{\partial}{\partial y_j}).\]
	
	The last identities follows from a direct computation, Lemma \ref{basicnoc} and Lemma \ref{lem1}.
\end{proof}

The preceding lemma allows a discussion on real tangent directions. In particular, if $\frac{\partial}{\partial y_j}$ is transversal to the set of degenerate Levi-forms $\Sigma$, one can easily construct $\psi$ along the direction of $\frac{\partial}{\partial y_j}$, so that \[g(\nabla_\nu\nu, \frac{\partial}{\partial y_j})=\frac{\partial}{\partial y_j}\psi\qquad\text{and}\qquad\frac{\partial}{\partial y_j}g(\nabla_\nu\nu,\frac{\partial}{\partial y_j})=\frac{\partial}{\partial y_j}(\frac{\partial}{\partial y_j}\psi).\] This happens even if $\frac{\partial}{\partial x_j}$ is tangent to $\Sigma$. Thus, a real tangent vector transversal to $\Sigma$ is not an obstruction. By the discussion above, we obtain following theorems.

As in Section \ref{sec4}, we say $\mathcal{H}^1(\Sigma)$ is \textit{trivial} if any closed form in $\Sigma$ which can be extended to a neighborhood of $\Sigma$ is exact.
\begin{theorem}\label{smooth}
	Let $\Omega$ be a bounded pseudoconvex domain with smooth boundary in $\mathbb{C}^n$. Let $\Sigma$ be the set of points with degenerate Levi-forms. Assume $\Sigma$ is a smooth manifold with boundary of dimension $m$ (with induced metric) and the Levi-form $\Hessian_\delta(Z_p, Z_p)\vert_{p\in\Sigma}=0$ if $X_p\in T_p\Sigma$, where $Z_p=\frac{1}{2}(X_p-\i J(X_p))$. Let the first de Rham cohomology $\mathcal{H}^1(\Sigma)$ be trivial. Then the Diederich--Forn\ae ss index of $\Omega$ is 1.
\end{theorem}

In $\mathbb{C}^2$, if the set $\Sigma$ of Levi-flat points forms a real curve which is transversal to $T^{(1,0)}\partial\Omega$, then the Diederich--Forn\ae ss index is 1. This result was proved by Krantz--Liu--Peloso in \citep{KLP16}. We now show that even if the real curve is not transversal to $T^{(1,0)}\partial\Omega$, the index is also 1. Here, we call a simple curve (1-manifold) a \textit{real curve} if it can be parametrized by a smooth map $\gamma: t\mapsto\partial\Omega$, where $t\in[0,1]$.

\begin{theorem}\label{newtheorem}
	Let $\Omega$ be a bounded pseudoconvex domain with smooth boundary in $\mathbb{C}^2$. Let $\Sigma$ be the set of Levi-flat points. Suppose $\Sigma$ forms a real curve. Then the Diederich--Forn\ae ss index is 1.
\end{theorem}

\begin{proof}
	We just need to show that the case that $\frac{\partial}{\partial t}$ is parallel to $T^{(1,0)}\partial\Omega$.  By the tubular neighborhood theorem, there is a tubular neighborhood of $\Sigma$ in $\partial\Omega$. We are going to find $\psi$ defined on this tubular neighborhood. By Theorem \ref{pre}, we notice that if we can show the following inequality
	\begin{equation}\label{new}\begin{split}
				&\left(\frac{1}{1-\eta}-1\right)\left|\frac{\partial}{\partial t}\psi+\i (J\frac{\partial}{\partial t})\psi+g(\nabla_\nu\nu, \frac{\partial}{\partial t})+\i g(\nabla_\nu\nu, J\frac{\partial}{\partial t})\right|^2\\&+\left(\frac{\partial^2}{\partial t^2}\psi+(J\frac{\partial}{\partial t})^2\psi+\frac{\partial}{\partial t}g(\nabla_\nu\nu, \frac{\partial}{\partial t})+J\frac{\partial}{\partial t}g(\nabla_\nu\nu, J\frac{\partial}{\partial t})\right)\leq 0
	\end{split}\end{equation}
along $\Sigma$ for arbitrary $\eta\in(0,1)$, then we are done. For the proof, we always assume $\psi\equiv0$ on $\Sigma$. This gives the following identities,
	\[\frac{\partial}{\partial t}\psi=0\qquad\text{and}\qquad \frac{\partial^2}{\partial t^2}\psi=0.\]
	Since $J\frac{\partial}{\partial t}$ is transversal to $\Sigma$, we also can assume $\psi$ satisfies \[J\frac{\partial}{\partial t}\psi=-g(\nabla_\nu\nu,J\frac{\partial}{\partial t})\] along $\Sigma$. Thus (\ref{new}) becomes \[
	\left(\frac{1}{1-\eta}-1\right)\left|g(\nabla_\nu\nu, \frac{\partial}{\partial t})\right|^2+\left((J\frac{\partial}{\partial t})^2\psi+\frac{\partial}{\partial t}g(\nabla_\nu\nu, \frac{\partial}{\partial t})+J\frac{\partial}{\partial t}g(\nabla_\nu\nu, J\frac{\partial}{\partial t})\right)\leq 0.\]
	Then for each $\eta\in(0,1)$, we find a number $C_\eta>0$ by smoothness so that \[\left(\frac{1}{1-\eta}-1\right)\left|g(\nabla_\nu\nu, \frac{\partial}{\partial t})\right|^2+\left(\frac{\partial}{\partial t}g(\nabla_\nu\nu, \frac{\partial}{\partial t})+J\frac{\partial}{\partial t}g(\nabla_\nu\nu, J\frac{\partial}{\partial t})\right)\leq C_\eta.\] Again, since $J\frac{\partial}{\partial t}$ is transversal to $\Sigma$, we can choose $\psi$ satisfying $(J\frac{\partial}{\partial t})^2\psi<-2C_\eta$. This gives that \[	\left(\frac{1}{1-\eta}-1\right)\left|g(\nabla_\nu\nu, \frac{\partial}{\partial t})\right|^2+\left((J\frac{\partial}{\partial t})^2\psi+\frac{\partial}{\partial t}g(\nabla_\nu\nu, \frac{\partial}{\partial t})+J\frac{\partial}{\partial t}g(\nabla_\nu\nu, J\frac{\partial}{\partial t})\right)\leq -C_\eta\leq 0.\] This completes the proof.
\end{proof}	

	\bigskip
	\bigskip
	\noindent {\bf Acknowledgments}. The author thanks to Dr. Steven Krantz for drawing his attention to this topic. The author thanks to Dr. Siqi Fu for his advice. The author also thank to Dr. Masanori Adachi, Dr. Xinghong Pan, Dr. Marco Peloso, Dr. Lihan Wang, Dr. Bun Wong, Dr. Yuan Yuan, Dr. Qi S. Zhang and Dr. Meng Zhu for fruitful conversations. Last but not least, the author thank to Dr. Jeffery McNeal for helping make the references more accurate.
	
	\printbibliography
	
\end{document}